\documentclass[11pt,leqno]{amsart}
\usepackage{amsmath,amssymb,latexsym}
\usepackage{amscd}
\usepackage[OT4]{fontenc}
\usepackage{lmodern}
\usepackage{mathrsfs}
\usepackage{amsthm}
\usepackage{amsfonts}
\usepackage{url}
\usepackage{enumerate}
\usepackage{graphicx}
\usepackage[all]{xy}
\newcommand{\R}{\mathbb R}
\newcommand{\N}{\mathbb N}
\newcommand{\Z}{\mathbb Z}

\newcommand{\C}{\mathbb C}

\newcommand{\cC}{\mathcal C}

\newcommand{\cI}{\mathcal I}
\newcommand{\cK}{\mathcal K}
\newcommand{\cL}{\mathcal L}

\newcommand{\cM}{\mathcal M}
\newcommand{\cN}{\mathcal N}
\newcommand{\cS}{\mathcal S}
\newcommand{\cX}{\mathcal X}
\newcommand{\im}{\operatorname{im}}
\newcommand{\alg}{\operatorname{alg}}

\newcommand{\lin}{\operatorname{lin}}
\newcommand{\HS}{\mathcal{HS}(\ell_2)}
\newcommand{\ldss}{\mathcal{L}(s',s)}

\newcommand{\lstars}{\mathcal{L}^*(s)}

\newcommand{\bproof}{{\raggedright\textbf{Proof.}} \ }

\newcommand{\bqed}{\hspace*{\fill} $\Box $\medskip}

\let\epsilon\varepsilon
\let\phi\varphi

\theoremstyle{plain}
\newtheorem{Th}{Theorem}[section]
\newtheorem{Cor}[Th]{Corollary}
\newtheorem{Lem}[Th]{Lemma}
\newtheorem{Prop}[Th]{Proposition}

\newtheorem{Exam}[Th]{Example}

\theoremstyle{definition}
\newtheorem{Def}[Th]{Definition}

\theoremstyle{remark}
\newtheorem{Rem}[Th]{Remark}

\title{{\sc Commutative subalgebras of the algebra of smooth operators}}
\author{{\sc Tomasz Cia{\'s}}}
\date{}

\begin{document}

\begin{abstract}
We consider the Fr\'echet ${}^*$-algebra $\ldss\subseteq\cL(\ell_2)$ of the so-called smooth operators, i.e.  continuous linear operators from the dual $s'$ of the space $s$ of rapidly decreasing sequences 
into $s$. This algebra is a non-commutative analogue of the algebra $s$. We characterize all closed commutative \mbox{${}^*$-subalgebras} of $\ldss$ which are at the same time isomorphic to closed 
${}^*$-subalgebras of $s$ and we provide an example of a closed commutative ${}^*$-subalgebra of $\ldss$ which cannot be embedded into $s$.
\end{abstract}

\maketitle

\footnotetext[1]{


{The research of the author was supported by the National Center of Science, grant no. 2013/09/N/ST1/04410}.}

\section{Introduction}
The algebra $\ldss$ can be represented as the algebra
\[\cK_\infty:=\{(x_{j,k})_{j,k\in\N}\in\C^{\N^2}:\sup_{j,k\in\N}|x_{j,k}|j^qk^q<\infty\textrm{ for all }q\in\N_0\}\] 
of rapidly decreasing matrices (with matrix multiplication and matrix complex conjugation).
Another representation of $\ldss$ is the algebra $\mathcal S(\R^2)$ of Schwartz functions on $\R^2$ with the Volterra convolution 
\[(f\cdot g)(x,y):=\int_\R f(x,z)g(z,y)\mathrm{d}z\] 
as multiplication and the involution
\[f^*(x,y):=\overline{f(y,x)}.\] 
In these forms, the algebra $\ldss$ usually appears and plays a significant role in $K$-theory of Fr\'echet algebras 
(see Bhatt and Inoue \cite[Ex. 2.12]{BhIO}, Cuntz \cite[p. 144]{Cuntz}, \cite[p. 64--65]{Cuntz2}, Gl\"ockner and Langkamp \cite{GlockLkamp}, 
Phillips \cite[Def. 2.1]{Phill}) and in $C^*$-dynamical systems (Elliot, Natsume and Nest \cite[Ex. 2.6]{ElNatNest}). Very recently, Piszczek obtained several rusults concerning
closed ideals, automatic continuity (for positive functionals and derivations), amenability and Jordan decomposition in $\cK_\infty$ 
(see Piszczek \cite{Pisz4,Pisz1} and his forthcoming papers `Automatic continuity and amenability in the non-commutative Schwartz space' and `The noncommutative Schwartz space is weakly amenable'). 
Moreover, in the context of algebras of unbounded operators, the algebra $\ldss$ appears in the book \cite{Sch} as 
\[\mathbb{B}_1(s):=\{x\in\cL(\ell_2)\colon x\ell_2\subseteq s, x^*\ell_2\subseteq s \text{ and } \overline{axb} \text{ is nuclear for all }a,b\in\lstars\},\]
where $\cL^*(s)$ is the so-called maximal $O^*$-algebra on $s$ (see also \cite[Def. 2.1.6, Prop. 2.1.8, Def. 5.1.3, Cor. 5.1.18, Prop. 5.4.1 and Prop. 6.1.5]{Sch}).

The algebra of smooth operators can be seen as a noncommutative analogue of the commutative algebra $s$.
The most important features of this algebra are the following:
\begin{itemize}
 \item it is isomorphic as a Fr\'echet space to the Schwartz space $\mathcal{S}(\R)$ of smooth rapidly decreasing functions on the real line;
 \item it has several representations as algebras of operators acting between natural spaces of distributions and functions (see \cite[Th. 1.1]{Dom});
 \item it is a dense ${}^*$-subalgebra of the $C^*$-algebra $\cK(\ell_2)$ of compact operators on $\ell_2$;
 \item it is (properly) contained in the intersection of all Schatten classes $\cS_p(\ell_2)$ over $p>0$; 
 in particular $\ldss$ is contained  the class $\HS$ of Hilbert-Schmidt operators, and thus it is a unitary space; 
 \item the operator $C^*$-norm $||\cdot||_{\ell_2\to\ell_2}$ is the so-called dominating norm on that algebra 
 (the dominating norm property is a key notion in the structure theory of nulcear Fr\'echet spaces -- see \cite[Prop. 3.2]{Cias} and \cite[Prop. 31.5]{MeV}). 
\end{itemize}

The main result of the present paper is a characterization of closed ${}^*$-subalgebras of $\ldss$ which are at the same time isomorphic as Fr\'echet ${}^*$-algebras to closed ${}^*$-subalgebras of $s$ 
(Theorem \ref{th_complemented_subalg}). It turns out that these are exactly those subalgebras which satisfy the classical condition $(\Omega)$ of Vogt. Then in Theorem \ref{th_counterexample} we give an example of a 
closed commutative ${}^*$-subalgebra of $\ldss$ which does not satisfy this condition. 

In order to prove this result we characterize in Section \ref{sec_kothe_algebras} closed ${}^*$-subalgebras of K\"othe sequence algebras (Proposition \ref{prop_subalg_kothe_alg}).
In particular, we give such a description for closed \mbox{${}^*$-subalgebras} of $s$ (Corollary \ref{cor_subalg_s2}).
In Section \ref{sec_kothe_repr} we describe all closed ${}^*$-subalgebras of $\ldss$ as suitable K\"othe sequence algebras (see Corollary \ref{cor_comm_subalg_max|f_j|2} and compare with \cite[Th. 4.8]{Cias}) 

The present paper is a continuation of \cite{Cias} and \cite{Dom} and it focuses on description of closed commutative 
${}^*$-subalgebras of $\ldss$ (especially those with the property ($\Omega$)). Most of the results have been already presented in the author PhD dissertation \cite{Cias_phd}.

\section{Notation and terminology}

Throughout the paper, $\N$ will denote the set of natural numbers $\{1,2,\ldots\}$ and $\N_0:=\N\cup\{0\}$.

By a \emph{projection}\index{projection} on the complex separable Hilbert space $\ell_2$ 
we always mean a continuous orthogonal (i.e. self-adjoint) projection. 

By $e_k$ we denote the vector in $\C^\N$ whose $k$-th coordinate equals 1 and the others equal 0.

By a \emph{Fr\'echet space} we mean a complete metrizable locally convex space over $\C$ (we will not use locally
convex spaces over $\R$). A \emph{Fr\'echet algebra} is a Fr\'echet space which is an algebra with continuous multiplication. 
A \emph{Fr\'echet ${}^*$-algebra} is a Fr\'echet algebra with continuous involution.

For locally convex spaces $E,F$, we denote by $\cL(E,F)$ the space of all continuous linear operators from $E$ to $F$. 
To shorten notation, we write $\cL(E)$ instead of $\cL(E,E)$.

We use the standard notation and terminology. All the notions from functional analysis are explained in \cite{Conw} or \cite{MeV} 
and those from topological algebras in \cite{Fra} or \cite{Zel}.

\section{Preliminaries}\label{preliminaries}

\paragraph{\textsc{The space $s$ and its dual}.}\label{par_s_s'}

We recall that the \emph{space of rapidly decreasing sequences}\index{space!of rapidly decreasing sequences} is the Fr\'echet space
\[s:=\bigg\{\xi=(\xi_j)_{j\in\N}\in\C^\N:|\xi|_q:=\bigg(\sum_{j=1}^\infty|\xi_j|^2j^{2q}\bigg)^{1/2}<\infty\text{ for all }q\in\N_0\bigg\}\]
with the topology corresponding to the system $(|\cdot|_q)_{q\in\N_0}$ of norms. 
We may identify the strong dual of $s$ (i.e. the space of all continuous linear functionals on $s$ with the topology of uniform convergence on bounded
subsets of $s$, see e.g. \cite[Def. on p. 267]{MeV}) with the \emph{space of slowly increasing sequences}\index{space!of slowly increasing sequences} 
\[s':=\bigg\{\xi=(\xi_j)_{j\in\N}\in\C^\N:|\xi|_q':=\bigg(\sum_{j=1}^\infty|\xi_j|^2j^{-2q}\bigg)^{1/2}<\infty\text{ for some }q\in\N_0\bigg\}\]
equipped with the inductive limit topology given by the system $(|\cdot|_q')_{q\in\N_0}$ of norms (note that for a fixed $q$, $|\cdot|'_q$ is defined
only on a subspace of $s'$).
More precisely, every $\eta\in s'$ corresponds to the continuous linear functional on $s$:
\[\xi\mapsto \langle \xi,\eta \rangle:=\sum_{j=1}^\infty \xi_j\overline{\eta_j}\]
(note the conjugation on the second variable).
These functionals are continuous, because, by the Cauchy-Schwartz inequality, for all $q\in\N_0$, $\xi\in s$ and $\eta\in s'$ we have
\begin{equation*}\label{eq_langle_xi_eta_rangle}
|\langle \xi,\eta\rangle|\leq |\xi|_q|\eta|'_q. 
\end{equation*} 
Conversely, one can show that for each continuous linear functional $y$ on $s$ there is $\eta\in s'$ such that $y=\langle\cdot,\eta\rangle$.
\label{functional_on_s} 

Similarly, we identify each $\xi\in s$ with the continuous linear functional on $s'$: 
\begin{equation*}\label{eq_langle_eta_xi_rangle}
\eta\mapsto\langle\eta,\xi\rangle:=\sum_{j=1}^\infty \eta_j\overline{\xi_j}. 
\end{equation*}
In particular, for each continuous linear functional $y$ on $s'$ there is $\xi\in s$ such that $y=\langle\cdot,\xi\rangle$.

We emphasize that the "scalar product" $\langle\cdot,\cdot\rangle$ 
is well-defined on $s\times s'\cup s'\times s$ and, of course, on $\ell_2\times\ell_2$.   

\paragraph{\textsc{Property (DN) for the space $s$ }.}\label{par_DN_s}
Closed subspaces of the space $s$ can be characterized by the so-called property (DN). 
\begin{Def}\label{def_DN}
A Fr\'echet space $(X,(||\cdot||_q)_{q\in\N_0})$ has the 
\emph{property} (DN) (see \cite[Def. on p. 359]{MeV}) if there is a continuous norm $||\cdot||$ on $X$ such that 
for all $q\in\N_0$ there is $r\in\N_0$ and $C>0$ such that 
\[||x||_q^2\leq C||x||\;||x||_r\]
for all $x\in X$. The norm $||\cdot||$ is called a \emph{dominating norm}.
\end{Def}

Vogt (see \cite{V1} and \cite[Ch. 31]{MeV}) proved that
a Fr\'echet space is isomorphic to a closed subspace of $s$ if and only if it is nuclear and it has the property (DN).

The (DN) condition for the space $s$ reads as follows (see \cite[Lemma 29.2(3)]{MeV} and its proof).  
\begin{Prop}\label{prop_l_2_DN_in_s}
For every $p\in\N_0$ and $\xi\in s$ we have
\[|\xi|_p^2\leq||\xi||_{\ell_2}|\xi|_{2p}.\]
In particular, the norm $||\cdot||_{\ell_2}$ is a dominating norm on $s$.
\end{Prop}

\paragraph{\textsc{The algebra $\ldss$}.}\label{par_algebra_ldss}
  
It is a simple matter to show that $\ldss$ with the topology of uniform convergence on bounded sets in $s'$ is a Fr\'echet space. It is isomorphic to
$s\widehat\otimes s$, the completed tensor product of $s$ (see \cite[\S41.7 (5)]{Kot} and note that, $s$ being nuclear, there is only
one tensor topology), and thus $\ldss\cong s$ as Fr\'echet spaces 
(see e.g. \cite[Lemma 31.1]{MeV}).\label{ldss_cong_s} Moreover, it is easily seen that  
$(||\cdot||_q)_{q\in\N_0}$,
\[||x||_q:=\sup_{|\xi|_q^{'}\leq 1}{|x\xi|_q},\]
is a fundamental sequence of norms on $\ldss$.

Let us introduce multiplication and involution on $\ldss$. 
First observe that $s$ is a dense subspace of $\ell_2$, $\ell_2$ is a dense subspace of $s'$, and, moreover, the embedding maps 
$j_1\colon s\hookrightarrow\ell_2$, $j_2\colon\ell_2\hookrightarrow s'$ are continuous.
Hence, 
\begin{equation*}\label{eq_iota2}
\iota\colon\ldss\hookrightarrow\cL(\ell_2),\quad\iota(x):=j_1\circ x \circ j_2,
\end{equation*}
is a well-defined (continuous) embedding of $\ldss$ into the $C^*$-algebra
$\cL(\ell_2)$, and thus it is natural to define a multiplication on 
$\ldss$ by 
\[xy:=\iota^{-1}(\iota(x)\circ\iota(y)),\] 
i.e. 
\[xy=x\circ j\circ y,\] 
where $j:=j_2\circ j_1\colon s\hookrightarrow s'$. Similarly, an involution on $\ldss$ is defined by 
\[x^*:=\iota^{-1}(\iota(x)^*),\] 
where $\iota(x)^*$ is the hermitian adjoint of $\iota(x)$. 
One can show that these definitions are correct, i.e. $\iota(x)\circ\iota(y),\iota(x)^*\in\iota(\ldss)$ for all $x,y\in\ldss$ (see also \cite[p. 148]{Cias}).

From now on, we will identify $x\in\ldss$ and $\iota(x)\in\cL(\ell_2)$ (we omit $\iota$ in the notation).

A Fr\'echet algebra $E$ is called \emph{locally $m$-convex}\index{algebra!Fr\'echet!locally $m$-convex} if $E$ has a fundamental system of submultiplicative seminorms.   
It is well-known that $\ldss$ is locally $m$-convex (see e.g. \cite[Lemma 2.2]{Phill}), and moreover, the norms 
$||\cdot||_q$ are submultiplicative (see \cite[Prop. 2.5]{Cias}). This shows simultaneously that the multiplication introduced above is separately 
continuous, and thus, by \cite[Th. 1.5]{Zel}, it is jointly continuous. Moreover, by \cite[Cor. 16.7]{Fra}, the involution on $\ldss$ is continuous.
  
We may summarize this paragraph by saying that $\ldss$ is a noncommutative ${}^*$-subalgebra of the $C^*$-algebra $\cL(\ell_2)$ which is 
(with its natural topology) a locally $m$-convex Fr\'echet ${}^*$-algebra isomorphic as a Fr\'echet space to $s$.

\section{K\"othe algebras}\label{sec_kothe_algebras}

In this section we collect and prove some results on K\"othe algebras which are known for specialists but probably never published.   

\begin{Def}
A matrix $A=(a_{j,q})_{j\in\N,q\in\N_0}$ of non-negative numbers such that
\begin{enumerate}[\upshape(i)]
 \item for each $j\in\N$ there is $q\in\N_0$ such that $a_{j,q}>0$
 \item $a_{j,q}\leq a_{j,q+1}$ for $j\in\N$ and $q\in\N_0$
\end{enumerate}
is called a \emph{K\"othe matrix}. 

For $1\leq p<\infty$ and a K\"othe matrix $A$ we define the \emph{K\"othe space} 
\[\lambda^p(A):=\bigg\{\xi=(\xi_j)_{j\in\N}\in\C^\N:|\xi|_{p,q}:=\bigg(\sum_{j=1}^\infty|\xi_j|^pa_{j,q}|^p\bigg)^{1/p}<\infty\text{ for all }q\in\N_0\bigg\}\]
and for $p=\infty$
\[\lambda^\infty(A):=\bigg\{\xi=(\xi_j)_{j\in\N}\in\C^\N:|\xi|_{\infty,q}:=\sup_{j\in\N}|\xi_j|a_{j,q}<\infty\text{ for all }q\in\N_0\bigg\}\]
with the locally convex topology given by the seminorms $(|\cdot|_{p,q})_{q\in\N_0}$ (see e.g. \cite[Def. p. 326]{MeV}). 
\end{Def}

Sometimes, for simplicity, we will write $\lambda^\infty(a_{j,q})$ (i.e. only the entries of the matix) instead of $\lambda^\infty(A)$.  

It is well-known (see \cite[Lemma 27.1]{MeV}) that the spaces $\lambda^p(A)$ are Fr\'echet spaces and sometimes they are
Fr\'echet ${}^*$-algebras with pointwise multiplication and conjugation (e.g. if $a_{j,q}\geq1$ for all $j\in\N$ and $q\in\N_0$, see also \cite[Prop. 3.1]{Pir});
in that case they are called \emph{K\"othe algebras}.

Clearly, $s$ is the K\"othe space $\lambda^2(A)$ for $A=(j^q)_{j\in\N,q\in\N_0}$ and it is a Fr\'echet ${}^*$-algebra. 
Moreover, since the matrix $A$ satisfies the so-called
Grothendieck-Pietsch condition (see e.g. \cite[Prop. 28.16(6)]{MeV}), $s$ is nuclear, and thus it has also 
other K\"othe space representations (see again \cite[Prop. 28.16 \& Ex. 29.4(1)]{MeV}), i.e. for all $1\leq p\leq\infty$, $s=\lambda^p(A)$ as 
Fr\'echet spaces.

We use $\ell_2$-norms in the definition of $s$ to clarify our ideas, for example we have 
$|\xi|_0=||\xi||_{\ell_2}$ for $\xi\in s$ and $|\eta|_0'=||\eta||_{\ell_2}$ for $\eta\in\ell_2$. However, in some situations the supremum
norms $|\cdot|_{\infty,q}$ (as they are relatively easy to compute) or the $\ell_1$-norms will be more convenient.

\begin{Prop}\label{prop_isomorphisms_kothe_algebras}
Let $A=(a_{j,q})_{j\in\N,q\in\N_0}$, $B=(b_{j,q})_{j\in\N,q\in\N_0}$ be K\"othe matrices and for a bijection $\sigma\colon\N\to\N$
let $A_\sigma:=(a_{\sigma(j),q})_{j\in\N,q\in\N_0}$. Assume that $\lambda^\infty(A)$ and 
$\lambda^\infty(B)$ are Fr\'echet ${}^*$-algebras. 
Then the following assertions are equivalent:
\begin{enumerate}[\upshape (i)]
\item $\lambda^\infty(A)\cong\lambda^\infty(B)$ as Fr\'echet ${}^*$-algebras;
\item there is a bijection $\sigma\colon\N\to\N$ such that $\lambda^\infty(A_\sigma)=\lambda^\infty(B)$ 
as Fr\'echet ${}^*$-algebras;
\item there is a bijection $\sigma\colon\N\to\N$ such that $\lambda^\infty(A_\sigma)=\lambda^\infty(B)$ 
as sets;
\item there is a bijection $\sigma\colon\N\to\N$ such that
\begin{itemize}
 \item[($\alpha$)] $\forall q\in\N_0\;\exists r\in\N_0\;\exists C>0\;\forall j\in\N\quad a_{\sigma(j),q}\leq Cb_{j,r}$,
 \item[($\beta$)] $\forall r'\in\N_0\;\exists q'\in\N_0\;\exists C'>0\;\forall j\in\N\quad b_{j,r'}\leq C'a_{\sigma(j),q'}$.
\end{itemize}
\end{enumerate}
\end{Prop}
\bproof
(i)$\Rightarrow$(ii) Assume that there is an isomorphism $\Phi\colon\lambda^\infty(A)\to\lambda^\infty(B)$ of Fr\'echet ${}^*$-algebras. 
Clearly, if $\xi^2=\xi$, then $\Phi(\xi)=\Phi(\xi^2)=(\Phi(\xi))^2$,
and the same is true for $\Phi^{-1}$, i.e. $\Phi$ maps the idempotents of $\lambda^\infty(A)$ onto the idempotents of $\lambda^\infty(B)$.
Hence for a fixed $k\in\N$, there is $I\subset\N$ such that 
\[\Phi(e_{k})=e_I,\]
where $e_I$ is a sequence which has 1 on an index set $I\subset\N$ and 0 otherwise. 
Suppose that $|I|\geq2$ and let $j\in I$. Then $e_I=e_j+e_{I\setminus\{j\}}$, where $e_j\in \lambda^\infty(B)$ and 
$e_{I\setminus\{j\}}=e_I-e_j\in \lambda^\infty(B)$. Therefore, there are nonempty subsets $I_j,I_{j}'\subset\N$ such that $\Phi(e_{I_j})=e_j$ and 
$\Phi(e_{I_j'})=e_{I\setminus\{j\}}$. We have
\[e_{I_j}e_{I_j'}=\Phi^{-1}(e_j)\Phi^{-1}(e_{I\setminus\{j\}})=\Phi^{-1}(e_je_{I\setminus\{j\}})=0,\]
and thus $I_j\cap I_j'=\emptyset$. Consequently,
\[\Phi(e_k)=e_j+e_{I\setminus\{j\}}=\Phi(e_{I_j})+\Phi(e_{I_j'})=\Phi(e_{I_j\cup I_j'}),\]
whence $1=|\{k\}|=|I_j\cup I_j'|\geq2$, a contradiction. Hence $\Phi(e_k)=e_{n_k}$ for some $n_k\in\N$, i.e. 
for the bijection $\sigma\colon\N\to\N$ defined by $n_{\sigma(k)}:=k$ we have $\Phi(e_{\sigma(k)})=e_k$. 
Therefore, a Fr\'echet ${}^*$-isomorphism $\Phi$ is given by $(\xi_{\sigma(k)})_{k\in\N}\mapsto(\xi_k)_{k\in\N}$ 
for $(\xi_{\sigma(k)})_{k\in\N}\in\lambda^\infty(A)$, and thus $\lambda^\infty(A_\sigma)=\lambda^\infty(B)$ 
as Fr\'echet ${}^*$-algebras.

(ii)$\Rightarrow$(iii) Obvious.

(iii)$\Rightarrow$(iv) The proof follows from the observation that the identity map 
$\operatorname{Id}\colon\lambda^\infty(A_\sigma)\to\lambda^\infty(B)$ is continuous (use the closed graph theorem).

(iv)$\Rightarrow$(i) It is easy to see that $\Phi\colon \lambda^\infty(A)\to \lambda^\infty(B)$ defined by $e_{\sigma(k)}\mapsto e_k$ 
is an isomorphism of Fr\'echet ${}^*$-algebras.
\bqed

In the following proposition we characterize infinite-dimensional closed ${}^*$-subalgebras of nuclear K\"othe algebras whose elements tends to zero
(note that if a K\"othe space is contained in $\ell_\infty$ then it is a K\"othe algebra). 
Consequently, we obtain a characterization of closed ${}^*$-subalgebras of $s$ (Corollary \ref{cor_subalg_s2}). 

\begin{Prop}\label{prop_subalg_kothe_alg}
For $\cN\subset\N$ let $e_\cN$ denote a sequence which has 1 on $\cN$ and 0 otherwise. 
Let $A=(a_{j,q})_{j\in\N,q\in\N_0}$ be a K\"othe matrix such that $\lambda^\infty(A)$ is nuclear and $\lambda^\infty(A)\subset c_0$.
Let $E$ be an infinite-dimensional closed ${}^*$-subalgebra of $\lambda^\infty(A)$. Then 
\begin{enumerate}[\upshape(i)]
 \item there is a family $\{\cN_k\}_{k\in\N}$ of finite nonempty pairwise disjoint sets of natural numbers such that
 $(e_{\cN_k})_{k\in\N}$ is a Schauder basis of $E$;
 \item $E\cong\lambda^\infty\left(\max_{j\in\cN_k}a_{j,q}\right)$ as Fr\'echet ${}^*$-algebras and the isomorphism is given by
 $e_{\cN_k}\mapsto e_k$ for $k\in\N$.
\end{enumerate}

Conversely, if $\{\cN_k\}_{k\in\N}$ is a family of finite nonempty pairwise disjoint sets of natural numbers and $F$ is the closed ${}^*$-subalgebra 
of $\lambda^\infty(A)$ generated by the set $\{e_{\cN_k}\}_{k\in\N}$, then 
\begin{enumerate}
 \item[\emph{(iii)}] $(e_{\cN_k})_{k\in\N}$ is a Schauder basis of $F$;
 \item[\emph{(iv)}] $F\cong\lambda^\infty(\max_{j\in\cN_k}a_{j,q})$ as Fr\'echet ${}^*$-algebras and the isomorphism is given by
 $e_{\cN_k}\mapsto e_k$ for $k\in\N$.
\end{enumerate}
\end{Prop}
\bproof
In order to prove (i) and (ii) set 
\[\cN_0:=\{j\in\N\colon\xi_j=0\quad\text{for all $\xi\in E$}\}\]
and define an equivalence relation $\sim$ on $\N\setminus\cN_0$ by
\[i\sim j\Leftrightarrow\xi_i=\xi_j\text{ for all $\xi\in E$}.\]
Since $E$ is infinite-dimensional, our relation produces infinitely many equivalence classes $\cN_k$, say  
\[\cN_k:=[\min(\N\setminus\cN_0\cup\ldots\cup\cN_{k-1})]_{/\sim}\]
for $k\in\N$. 

Fix $\kappa\in\N$ and take $\xi\in E$ such that $\xi_j\neq 0$ for $j\in\cN_{\kappa}$. Denote $\eta_k:=\xi_j$ if $j\in\cN_k$. Let 
\[\cM_1:=\{j\in\N\colon |\eta_j|=\sup_{i\in\N}{|\eta_i|}\}.\]
Assume we have already defined $\cM_1,\ldots,\cM_{l-1}$. If there is $j\in\N\setminus\{\cM_1\cup\ldots\cup\cM_{l-1}\}$ such that $\eta_j\neq0$
then we define
\[\cM_l:=\{j\in\N\colon |\eta_j|=\sup\{|\eta_i|\colon i\in\N\setminus\cM_1\cup\ldots\cup\cM_{l-1}\}\}.\]
Otherwise, denote $\cI:=\{1,\ldots,l-1\}$. If this procedure leads to infinite many sets $\cM_l$ then we set $\cI:=\N$.
It is easily seen that for each $l\in\mathcal{I}$ there is $\mathcal{I}_l\subset\N$
such that $\cM_l=\bigcup_{k\in\mathcal{I}_l}\cN_k$.
By assumption $\xi\in c_0$, hence $(|\eta_k|)_{k\in\N}\in c_0$ as well, and thus each $\cM_l$ is a finite nonempty set.

We first show that $e_{\cM_l}\in E$ for $l\in\cI$. For $l\in\cI$ fix $m_l\in\cM_l$. If $\cI=\{1\}$, then $\xi_j=0$ for $j\notin\cM_1$, and 
$e_{\cM_1}=\frac{\xi\overline{\xi}}{|\eta_{m_1}|^2}\in E$. Let us consider the case $|\cI|>1$. 
Since in nuclear Fr\'echet spaces every basis is absolute (and thus unconditional), we have 
\[\sum_{l\in\cI}|\eta_l|^2 e_{\cM_l}=\sum_{j=1}^\infty|\xi_j|^2e_j=\xi\overline{\xi}\in E,\]
and, consequently,
\[x_n:=\sum_{l\in\cI}\bigg(\frac{|\eta_l|}{|\eta_{m_1}|}\bigg)^{2n}e_{\cM_l}=\bigg(\frac{\xi\overline{\xi}}{|\eta_{m_1}|^2}\bigg)^n\in E\]
for all $n\in\N$.
Then for $q$ and $n$ we get
\begin{align*}
|x_n-e_{\cM_1}|_{\infty,q}
&=\bigg|\sum_{l\in\cI}^\infty\bigg(\frac{|\eta_l|}{|\eta_{m_1}|}\bigg)^{2n}e_{\cM_l}-e_{\cM_1}\bigg|_{\infty,q}
=\bigg|\sum_{l\in\cI\setminus\{1\}}\bigg(\frac{|\eta_l|}{|\eta_{m_1}|}\bigg)^{2n}e_{\cM_l}\bigg|_{\infty,q}\\
&\leq\sum_{l\in\cI\setminus\{1\}}\bigg(\frac{|\eta_l|}{|\eta_{m_1}|}\bigg)^{2n}|e_{\cM_l}|_{\infty,q}
\leq\frac{1}{|\eta_{m_1}|}\bigg(\frac{|\eta_{m_2}|}{|\eta_{m_1}|}\bigg)^{2n-1}\sum_{l\in\cI\setminus\{1\}}|\eta_l|\:|e_{\cM_l}|_{\infty,q}.
\end{align*}
Since $(e_j)_{j\in\N}$ is an absolute basis in $\lambda^\infty(A)$, the above series is convergent. Note also that $|\eta_{m_2}|<|\eta_{m_1}|$. 
This shows that $x_n\to e_{\cM_1}$ in $\lambda^\infty(A)$, and $e_{\cM_1}\in E$. Assume that $e_{\cM_1},\ldots,e_{\cM_{l-1}}\in E$. 
If $|\cI|=l-1$ then we are done. Otherwise, $\eta_{m_l}\neq0$ and  
\[x_n^{(l)}:=\left(\frac{\xi\overline{\xi}-\xi\overline{\xi}\sum_{j=1}^{l-1}e_{\cM_j}}{|\eta_{m_l}|^2}\right)^n\in E\]
for $n\in\N$. As above we show that $x_n^{(l)}\to e_{\cM_l}$ in $\lambda^\infty(A)$, and thus $e_{\cM_l}\in E$. Proceeding by induction, we prove
that $e_{\cM_l}\in E$ for $l\in\cI$.

Now, we shall prove that $(e_{\cN_k})_{k\in\N}$ is a Schauder basis of $E$. Choose $\iota\in\cI$ such that $\kappa\in\cI_\iota$ and 
for $k\in\cI_{\iota}$ let $n_k$ be an arbitrary element of $\cN_k$. Then $\sum_{k\in\cI_\iota}\eta_{n_k}e_{\cN_k}=\xi e_{\cM_\iota}\in E$.
Consequently, by \cite[Lemma 4.1]{Cias}, $e_{\cN_\kappa}\in E$. Since $\kappa$ was arbitrarily choosen, each $e_{\cN_k}$ is in $E$ and 
it is a simple matter to show that $(e_{\cN_k})_{k\in\N}$ is a Schauder basis of $E$.

Moreover, $|e_{\cN_k}|_{\infty,q}=\max_{j\in\cN_k}{a_{j,q}}$ hence, by \cite[Cor. 28.13]{MeV} and nuclearity, 
$E$ is isomorphic as a Fr\'echet space to $\lambda^\infty(\max_{j\in\cN_k}{a_{j,q}})$. The analysis of the proof of \cite[Cor. 28.13]{MeV}
shows that this isomorphism is given by $e_{\cN_k}\mapsto e_k$ for $k\in\N$, and thus it is also a Fr\'echet ${}^*$-algebra isomorphism.

Now, we prove (iii) and (iv). First note that every element of $F$ is the limit of elements of the form $\sum_{k=1}^Mc_ke_{\cN_k}$, 
where $M\in\N$ and $c_1,\ldots,c_M\in\C$. Therefore, if $\xi\in F$, then $\xi_i=\xi_j$ for $k\in\N$ and $i,j\in\cN_k$. This shows that each $\xi\in F$
has the unique series representation $\xi=\sum_{k=1}^\infty \xi_{n_k}e_{\cN_k}$, where $(n_k)_{k\in\N}$ is an arbitrarily choosen sequence such that
$n_k\in\cN_k$ for $k\in\N$. Since the series is absolutely convergent, $(e_{\cN_k})_{k\in\N}$ is a Schauder basis of $F$. Statement (iv) 
follows by the same method as in (ii). 
\bqed

\begin{Cor}\label{cor_subalg_s2}
Every infinite-dimensional closed ${}^*$-subalgebra of $s$ is 
isomorphic as a Fr\'echet ${}^*$-algebra to $\lambda^\infty(n_k^q)$ for some strictly increasing sequence $(n_k)_{k\in\N}$ of natural numbers.
Conversely, if $(n_k)_{k\in\N}$ is a strictly increasing sequence of natural numbers, 
then $\lambda^\infty(n_k^q)$ is isomorphic as a Fr\'echet ${}^*$-algebra to some infinite-dimensional closed 
${}^*$-subalgebra of $s$.  
Moreover, every closed ${}^*$-subalgebra of $s$ is a complemented subspace of $s$.
\end{Cor}
\bproof
We apply Proposition \ref{prop_subalg_kothe_alg} to the K\"othe matrix $(j^q)_{j\in\N,q\in\N_0}$. 
Let $\{\cN_k\}_{k\in\N}$ be a family of finite nonempty pairwise disjoint sets of natural numbers. We have
\begin{equation}\label{eq_maxjq}
\max_{j\in\cN_k}j^q=(\max\{j:j\in\cN_k\})^q 
\end{equation}
for all $q\in\N_0$ and $k\in\N$.
Let $\sigma\colon\N\to\N$ be the bijection for which $(\max\{j:j\in\cN_{\sigma(k)}\})_{k\in\N}$ is (strictly) increasing and let 
$n_k:=\max\{j:j\in\cN_{\sigma(k)}\}$ for $k\in\N$. Then, by Proposition \ref{prop_isomorphisms_kothe_algebras},
\[\lambda^\infty\bigg(\max_{j\in\cN_k}j^q\bigg)\cong\lambda^\infty(n_k^q)\]
as Fr\'echet ${}^*$-algebras, and therefore the first two statements follow from Proposition \ref{prop_subalg_kothe_alg}.

Now, let $E$ be a closed ${}^*$-subalgebra of $s$. If $E$ is finite dimensional then, clearly, $E$ is complemented in $s$. Otherwise, by Proposition \ref{prop_subalg_kothe_alg}(i), 
$E$ is a closed linear span of the set $\{e_{\cN_k}\}_{k\in\N}$ for some family $\{\cN_k\}_{k\in\N}$ of finite nonempty pairwise disjoint sets of natural numbers. Define $\pi\colon s\to E$ by
\begin{displaymath}
(\pi x)_j:= \left\{ \begin{array}{ll}
x_{n_k} & \textrm{for $j\in\cN_{\sigma(k)}$}\\
0 & \textrm{otherwise}
\end{array} \right.
\end{displaymath}
where $(n_k)_{k\in\N}$ and $\sigma$ are as above. From (\ref{eq_maxjq}) we have for every $q\in\N_0$
\begin{align*}
|\pi x|_{\infty,q}&=\sup_{j\in\N}|(\pi x)_j|j^q\leq\sup_{k\in\N}|x_{n_k}|\max_{j\in\cN_{\sigma(k)}}j^q=\sup_{k\in\N}|x_{n_k}|(\max\{j:j\in\cN_k\})^q\\
&=\sup_{k\in\N}|x_{n_k}|n_k^q\leq\sup_{j\in\N}|x_j|j^q=|x|_{\infty,q},
\end{align*}
and thus $\pi$ is well-defined and continuous. Since $\pi$ is a projection, our proof is complete.

\bqed

\section{Representations of closed commutative ${}^*$-subalgebras of $\ldss$ by K\"othe algebras}\label{sec_kothe_repr} 

The aim of this section is to describe all closed commutative ${}^*$-subalgebras of $\ldss$ as K\"othe algebras $\lambda^\infty(A)$ for matrices $A$ determined by orthonormal sequences 
whose elements belong to the space $s$ (Theorem \ref{th_comm_subalg_f_n2} and Corollaries \ref{cor_comm_subalg_max|f_j|2} and \ref{cor_max_comm_subalg_f_n2}).
For the convenience of the reader, we quote two results from \cite{Cias} (with minor modifications which do not require extra arguments).

For a subset $Z$ of $\ldss$ we will denote by $\alg(Z)$ ($\overline{\operatorname{lin}}(Z)$, resp.) the closed ${}^*$-subalgebra of $\ldss$ 
generated by $Z$ (the closed linear span of $Z$, resp.).

By \cite[Lemma 4.4]{Cias}, every closed commutative ${}^*$-subalgebra $E$ of $\ldss$ admits a special Schauder basis. This basis consists of
all nonzero minimal projections in $E$ (\cite[Lemma 4.4]{Cias} shows that these projections are pairwise orthogonal) and we call it the \emph{canonical Schauder basis} of $E$.

\begin{Prop}\label{prop_P_n_basic_sequnce}\cite[Prop. 4.7]{Cias}
Every sequence $\{P_k\}_{k\in\cN}\subset\ldss$ of nonzero pairwise orthogonal projections is 
the canonical Schauder basis of the algebra $\alg(\{P_k\}_{k\in\cN})$.
In particular, $\{P_k\}_{k\in\cN}$ is a basic sequence in $\ldss$, i.e. 
it is a Schauder basis of the Fr\'echet space $\overline{\operatorname{lin}}(\{P_k\}_{k\in\cN})$.
\end{Prop}

\begin{Th}\label{th_commutative_subalg}\cite[Th. 4.8]{Cias}
Let $E$ be a closed commutative infinite-dimensional ${}^*$-subalgebra of $\ldss$ and let $\{P_k\}_{k\in\N}$ be the canonical Schauder basis  
of $E$. Then
\[E=\alg(\{P_k\}_{k\in\N})\cong\lambda^\infty(||P_k||_q)\] 
as Fr\'echet ${}^*$-algebras and the isomorphism is given by $P_k\mapsto e_k$ for $k\in\N$. 
\end{Th}

Please note that a projection $P\in\ldss$ if and only if it is of the form
\[P\xi=\sum_{k\in I}\langle\xi,f_k\rangle f_k\]
for some finite set $I$ and an orthonormal sequence $(f_k)_{k\in I}\subset s$.

We will also use the identity
\begin{equation}\label{eq_|f_k|=||f_k||}
\lambda^\infty(||\langle\cdot,f_k\rangle f_k||_q)=\lambda^\infty(|f_k|_q) 
\end{equation}
which holds for every orthonormal sequence $(f_k)_{k\in\N}\subset s$.
(see \cite[Rem. 4.11]{Cias}).  

Now we are ready to state and prove the main result of this section. 

\begin{Th}\label{th_comm_subalg_f_n2}
Every closed commutative ${}^*$-subalgebra of $\ldss$ is isomorphic as a Fr\'echet \mbox{${}^*$-algebra} to some closed \mbox{${}^*$-subalgebra}
of the algebra $\lambda^\infty(|f_k|_q)$ for some orthonormal sequence $(f_k)_{k\in\N}\subset s$. 
More precisely, if $E$ is an infinite-dimensional closed commutative ${}^*$-subalgebra of $\ldss$ and 
$(\sum_{j\in\cN_k}\langle\cdot,f_j\rangle f_j)_{k\in\N}$ is its canonical Schauder basis for some family of finite pairwise disjoint subsets 
$(\cN_k)_{k\in\N}$ of natural numbers and an orthonormal sequence $(f_j)_{j\in\N}\subset s$, 
then $E$ is isomorphic as a Fr\'echet ${}^*$-algebra to the closed ${}^*$-subalgebra of $\lambda^\infty(|f_k|_q)$ generated by 
$\{\sum_{j\in\cN_k}e_j\}_{k\in\N}$ and the isomorphism is given by $\sum_{j\in\cN_k}\langle\cdot,f_j\rangle f_j\mapsto \sum_{j\in\cN_k}e_j$ for $k\in\N$.

Conversely, if $(f_k)_{k\in\N}\subset s$ is an orthonormal sequence, 
then every closed ${}^*$-sub\-algebra 
of $\lambda^\infty(|f_k|_q)$ is isomorphic as a Fr\'echet ${}^*$-algebra to some closed commutative ${}^*$-subalgebra of $\ldss$.
\end{Th}

\bproof
By Theorem \ref{th_commutative_subalg}, $E=\alg\bigg(\bigg\{\sum_{j\in\cN_k}\langle\cdot,f_j\rangle f_j\bigg\}_{k\in\N}\bigg)$ for $(\cN_k)_{k\in\N}$ and $(f_j)_{j\in\N}\subset s$ as in the statement.
Let $F$ be the closed ${}^*$-subalgebra of $\lambda^\infty(|f_k|_q)$ generated by $\{\sum_{j\in\cN_k}e_j\}_{k\in\N}$.
Define 
\[\Phi\colon\alg(\{\langle\cdot,f_k\rangle f_k\}_{k\in\cN})\to\lambda^\infty(|f_k|_q)\] 
by $\langle\cdot,f_k\rangle f_k\mapsto e_k$, where $\cN:=\bigcup_{k\in\N}\cN_k$. By Proposition \ref{prop_P_n_basic_sequnce}, $\{\langle\cdot,f_k\rangle f_k\}_{k\in\cN}$ is the canonical 
Schauder basis of $\alg(\{\langle\cdot,f_k\rangle f_k\}_{k\in\cN})$, and thus Theroem \ref{th_commutative_subalg} and (\ref{eq_|f_k|=||f_k||}) imply that $\Phi$ is a Fr\'echet ${}^*$-algebra isomorphism. 
Hence, $(\sum_{j\in\cN_k}e_j)_{k\in\N}=(\Phi(\sum_{j\in\cN_k}\langle\cdot,f_j\rangle f_j))_{k\in\N}$ is a Schauder basis of $\Phi(E)$ and $\Phi(E)$ is a closed ${}^*$-subalgebra of $\lambda^\infty(|f_k|_q)$. 
Therefore,
\[\Phi(E)=\overline{\lin}\bigg(\bigg\{\sum_{j\in\cN_k}e_j\bigg\}_{k\in\N}\bigg)\subset F\subset\Phi(E),\]
whence $\Phi(E)=F$. In consequence $\Phi_{\mid E}$ is a Fr\'echet ${}^*$-algebra isomorphism of $E$ and $F$, which completes the proof of the first statement.

If now $(f_k)_{k\in\N}\subset s$ is an arbitrary orthonormal sequence then, according to Proposition \ref{prop_P_n_basic_sequnce}, 
Theorem \ref{th_commutative_subalg} and indentity (\ref{eq_|f_k|=||f_k||}), $\lambda^\infty(|f_k|_q)\cong\alg(\{\langle\cdot,f_k\rangle f_k\}_{k\in\N})$ as Fr\'echet ${}^*$-algebras. 
Consequently, every closed ${}^*$-subalgebra of $\lambda^\infty(|f_k|_q)$ is isomorphic as a Fr\'echet ${}^*$-algebra to some closed 
${}^*$-subalgebra of $\alg(\{\langle\cdot,f_k\rangle f_k\}_{k\in\N})$.
\bqed

The following characterization of infinite-dimensional closed commutative ${}^*$-subalgebras of $\ldss$ is a straightforward consequence
of Proposition \ref{prop_subalg_kothe_alg} and Theorem \ref{th_comm_subalg_f_n2}.
\begin{Cor}\label{cor_comm_subalg_max|f_j|2}
Every infinite-dimensional closed commutative ${}^*$-subalgebra of $\ldss$ is isomorphic as a Fr\'echet ${}^*$-algebra to the algebra 
$\lambda^\infty(\max_{j\in\cN_k}|f_j|_q)$ for some orthonormal sequence $(f_k)_{k\in\N}\subset s$ and some family $\{\cN_k\}_{k\in\N}$ of finite 
nonempty pairwise disjoint sets of natural numbers. 
In fact, if $E$ is an infinite-dimensional closed commutative ${}^*$-subalgebra of $\ldss$ and $(\sum_{j\in\cN_k}\langle\cdot,f_j\rangle f_j)_{k\in\N}$ 
is its canonical Schauder basis, then 
\[E\cong\lambda^\infty\bigg(\max_{j\in\cN_k}|f_j|_q\bigg)\]
as Fr\'echet ${}^*$-algebras and the isomorphism is given by $\sum_{j\in\cN_k}\langle\cdot,f_j\rangle f_j\mapsto e_k$ for $k\in\N$.

Conversely, if $(f_k)_{k\in\N}\subset s$ is an orthonormal sequence and $\{\cN_k\}_{k\in\N}$ is a family of finite 
nonempty pairwise disjoint sets of natural numbers, 
then $\lambda^\infty(\max_{j\in\cN_k}|f_j|_q)$ is isomorphic as a Fr\'echet ${}^*$-algebra to some infinite-dimensional closed commutative 
${}^*$-subalgebra of $\ldss$. 
\end{Cor}

At the end of this section we consider the case of maximal commutative subalgebras of $\ldss$.
A closed commutative ${}^*$-subalgebra of $\ldss$ is said to be \emph{maximal commutative}
if it is not properly contained in any larger closed commutative \mbox{${}^*$-subalgebra} of $\ldss$.

We say that an orthonormal system $(f_k)_{k\in\N}$ of $\ell_2$ is \emph{$s$-complete}\index{$s$-complete sequence}, if every $f_k$ belongs to $s$ 
and for every $\xi\in s$ the following implication holds: if $\langle \xi,f_k\rangle=0$ for every $k\in\N$, then $\xi=0$. 
A sequence $\{P_k\}_{k\in\N}$ of nonzero pairwise orthogonal projections belonging to $\ldss$ 
is called $\ldss$-\emph{complete} if there is no nonzero projection $P$ belonging to $\ldss$ such that $P_kP=0$ 
for every $k\in\N$.

One can easily show that an orthonormal 
system $(f_k)_{k\in\N}$ is $s$-complete if and only if the sequence of projections $(\langle\cdot,f_k\rangle f_k)_{k\in\N}$ is $\ldss$-complete.
Hence, by \cite[Th. 4.10]{Cias}, closed commutative ${}^*$-subalgebra $E$ of $\ldss$ is maximal commutative if and only if there is an $s$-complete sequence $(f_k)_{k\in\N}$ such that 
$(\langle\cdot,f_k\rangle f_k)_{k\in\N}$ is the canonical Schauder basis of $E$. Combining this with Corollary \ref{cor_comm_subalg_max|f_j|2}, we obtain the first statement of the 
following Corollary.

\begin{Cor}\label{cor_max_comm_subalg_f_n2}
Every closed maximal commutative ${}^*$-subalgebra of $\ldss$ is isomorphic as a Fr\'echet ${}^*$-algebra to the algebra 
$\lambda^\infty(|f_k|_q)$ for some $s$-complete orthonormal sequence $(f_k)_{k\in\N}$. 
More precisely, if $E$ is a closed maximal commutative ${}^*$-subalgebra of $\ldss$ with the canonical Schauder basis 
$(\langle\cdot,f_k\rangle f_k)_{k\in\N}$, then 
\[E\cong\lambda^\infty(|f_k|_q)\]
as Fr\'echet ${}^*$-algebras and the isomorphism is given by $\langle\cdot,f_k\rangle f_k\mapsto e_k$ for $k\in\N$.

Conversely, if $(f_k)_{k\in\N}$ is an $s$-complete orthonormal sequence, 
then $\lambda^\infty(|f_k|_q)$ is isomorphic as a Fr\'echet ${}^*$-algebra to some closed maximal commutative ${}^*$-subalgebra of $\ldss$. 

\end{Cor}

\bproof
In order to prove the second statement, take an arbitrary $s$-complete orthonormal sequence $(f_k)_{k\in\N}$. 
By Proposition \ref{prop_P_n_basic_sequnce} and the remark above our Corollary, $\alg(\{\langle\cdot, f_k\rangle f_k\}_{k\in\N})$ is maximal commutative and from the first statement it follows that it is
isomorphic as a Fr\'echet ${}^*$-algebra to $\lambda^\infty(|f_k|_q)$.   
\bqed

It is also worth pointing out the following result.

\begin{Prop}\label{prop_alg_extended_to_maximal}
Every closed commutative ${}^*$-subalgebra of $\ldss$ is contained in some maximal commutative ${}^*$-subalgebra of $\ldss$. 
\end{Prop}
\bproof
Let $E$ be a closed commutative ${}^*$-subalgebra of $\ldss$.
Clearly, 
\[\cX:=\{\widetilde{E}: \widetilde{E} \text{ commutative ${}^*$-subalgebra of $\ldss$ and $E\subset\widetilde{E}$}\}\] 
with the inclusion relation is a partially ordered set. Consider
a chain $\cC$ in $\cX$ and let $E_{\cC}:=\bigcup_{F\in\cC}F$. It is easy to check that $E_{\cC}\in\cX$, and, of course, 
$E_{\cC}$ is an upper bound of $\cC$. Hence, by the Kuratowski-Zorn lemma, $\cX$ has a maximal element; let us call it $M$. 
By the continuity of the algebra operations, $\overline{M}^{\ldss}$ is a closed commutative ${}^*$-subalgebra of $\ldss$, hence from the maximality of $M$, 
we have $M=\overline{M}^{\ldss}$, i.e. $M$ is a (closed) maximal commutative ${}^*$-subalgebra of $\ldss$ containing $E$. 
\bqed

\section{Closed commutative ${}^*$-subalgebras of $\ldss$ with the property ($\Omega$)}\label{sec_omega}
In the present section we prove that a closed commutative ${}^*$-subalgebra of $\ldss$ is isomorphic as a Fr\'echet ${}^*$-algebra to
some closed ${}^*$-subalgebra of $s$ if and only if it is isomorphic as a Fr\'echet space to some complemented subspace of $s$ 
(Theorem \ref{th_complemented_subalg}), i.e. if it has the so-called property ($\Omega$) (see Definition \ref{def_Omega} below).
We also give an example of a closed commutative ${}^*$-subalgebra of $\ldss$ which is not isomorphic to any closed ${}^*$-subalgebra of $s$
(Theorem \ref{th_counterexample}).

\begin{Def}\label{def_Omega}\index{property!($\Omega$)}
A Fr\'echet space $E$ with a fundamental sequence $(||\cdot||_q)_{q\in\N_0}$ of seminorms has the \emph{property} ($\Omega$) if the following condition holds:
\[\forall p\;\exists q\;\forall r\;\exists\theta\in(0,1)\;\exists C>0\;\forall y\in E'\quad ||y||_q'\leq C||y||_p'^{1-\theta}||y||_r'^{\theta},\]
where $E'$ is the topological dual of $E$ and $||y||_p':=\sup\{|y(x)|:||x||_p\leq 1\}$. 
\end{Def}
The property ($\Omega$) (together with the property (DN)) plays a crucial role in the theory of nuclear Fr\'echet spaces 
(for details, see \cite[Ch. 29]{MeV}).

Recall that a subspace $F$ of a Fr\'echet space $E$ is called \emph{complemented} (in $E$) if there is a continuous projection 
$\pi\colon E\to E$ with $\im\pi=F$.
Since every subspace of $\ldss$ has the property (DN) (and, by \cite[Prop. 3.2]{Cias}, the norm $||\cdot||_{\ell_2\to\ell_2}$ 
is already a dominating norm), \cite[Prop. 31.7]{MeV} implies that a closed ${}^*$-subalgebra of $\ldss$ is isomorphic to 
a complemented subspace of $s$ if and only if it has the property ($\Omega$). 
The class of complemented subspaces of $s$ is still not well-understood (e.g. we do not know whether every such subspace has a Schauder basis -- the Pe{\l}czy{\'n}ski problem)
and, on the other hand, the class of closed ${}^*$-subalgebras of $s$ has a simple description (see Corollary \ref{cor_subalg_s2}). The following theorem implies that, when restricting to the
family of closed commutative ${}^*$-subalgebras of $\ldss$, these two classes of Fr\'echet spaces coincide.

\begin{Th}\label{th_complemented_subalg}
Let $E$ be an infinite-dimensional closed commutative ${}^*$-subalgebra of $\ldss$ and let 
$(\sum_{j\in\cN_k}\langle\cdot,f_j\rangle f_j)_{k\in\N}$ be its canonical Schauder basis. Then the following assertions are equivalent:
\begin{enumerate}[\upshape(i)]
 \item $E$ is isomorphic as a Fr\'echet ${}^*$-algebra to some closed ${}^*$-subalgebra of $s$;
 \item $E$ is isomorphic as a Fr\'echet space to some complemented subspace of $s$;
 \item $E$ has the property \emph{($\Omega$)};
 \item $\exists p\;\forall q\;\exists r\;\exists C>0\;\forall k\quad\max_{j\in\cN_k}|f_j|_q\leq C\max_{j\in\cN_k}|f_j|_p^r$.
\end{enumerate}
\end{Th}

In order to prove Theorem \ref{th_complemented_subalg}, we will need Lemmas \ref{lem_f_k_nuclearity}, \ref{lem_a_k} and Propositions 
\ref{prop_characterization_commutative*-subalgebras}, \ref{prop_l_2_DN_in_s2}.

The following result is a consequence of nuclearity of closed commutative ${}^*$-subalgebras of $\ldss$. 
\begin{Lem}\label{lem_f_k_nuclearity}
Let $(f_k)_{k\in\N}\subset s$ be an orthonormal sequence and let $(\cN_k)_{k\in\N}$ be a family of finite pairwise disjoint subsets 
of natural numbers. For $r\in\N_0$ let $\sigma_r\colon\N\to\N$ be a bijection such that the sequence 
$(\max_{j\in\cN_{\sigma_r(k)}}|f_j|_r)_{k\in\N}$ is non-decreasing. Then there is $r_0\in\N$ such that  
 \[\lim_{k\to\infty}\frac{k}{\max_{j\in\cN_{\sigma_r(k)}}|f_j|_r}=0\]
for all $r\geq r_0$.
\end{Lem}
\begin{proof}
By Corollary \ref{cor_comm_subalg_max|f_j|2}, $\lambda^\infty(\max_{j\in\cN_k}|f_j|_q)$ is a nuclear space. Hence, by the Grothendieck-Pietsch theorem 
(see e.g. \cite[Th. 28.15]{MeV}), for every $q\in\N_0$ there is $r\in\N_0$ such that 
\[\sum_{k=1}^\infty\frac{\max_{j\in\cN_k}|f_j|_q}{\max_{j\in\cN_k}|f_j|_{r}}<\infty.\]
In particular (for $q=0$), there is $r_0$ such that for $r\geq r_0$ we have 
\[\sum_{k=1}^\infty\frac{1}{\max_{j\in\cN_{\sigma_r(k)}}|f_j|_r}=\sum_{k=1}^\infty\frac{1}{\max_{j\in\cN_k}|f_j|_r}<\infty.\] 
Since the sequence $(\max_{j\in\cN_{\sigma_r(k)}}|f_j|_r)_{k\in\N}$ is non-decreasing, 
the conclusion follows from the elementary theory of number series.
\end{proof}

\begin{Lem}\label{lem_a_k}
Let $(a_k)_{k\in\N}\subset[1,\infty)$ be a non-decreasing sequence such that $a_k\geq 2k$ for $k$ big enough. Then there exist
a strictly increasing sequence $(b_k)_{k\in\N}$ of natural numbers and $C>0$ such that
\[\frac{1}{C}a_k\leq b_k\leq Ca_k^2\]
for every $k\in\N$.
\end{Lem}
\bproof
Let $k_0\in\N$ be such that $a_k\geq2k$ for $k> k_0$ and choose $C\in\N$ so that
\[\frac{1}{C}a_k\leq k\leq Ca_k^2\]
for $k\in\cN_0:=\{1,\ldots,k_0\}$. Denote also $\cN_1:=\{k\in\N: a_k=a_{k_0+1}\}$ and, recursively,
$\cN_{j+1}:=\{k\in\N: a_k=a_{\max\cN_j+1}\}$. Clearly, $\cN_j$ are finite, pairwise disjoint, $\bigcup_{j\in\N_0}\cN_j=\N$
and $k<l$ for $k\in\cN_j$, $l\in\cN_{j+1}$. 

Let $b_k:=k$ for $k\in\cN_0$ and let 
\[b_{m_j+l-1}:=C\lceil\max\{a_{m_j-1}^2,a_{m_j}\}\rceil+l\]
for $j\in\N$ and $1\leq l\leq|\cN_j|$, where $m_j:=\min\cN_{j}$ and $\lceil x\rceil:=\min\{n\in\Z:n\geq x\}$ stands for the ceiling of $x\in\R$. 
We will show inductively that $(b_k)_{k\in\N}$ is a strictly increasing sequence of natural numbers such that
\begin{equation}\label{eq_C}
\frac{1}{C}a_k\leq b_k\leq Ca_k^2
\end{equation}
for every $k\in\N$. 

Clearly, the condition (\ref{eq_C}) holds for $k\in\cN_0$. Assume that $(b_k)_{k\in\cN_0\cup\ldots\cup\cN_j}$  
is a strictly increasing sequence of natural numbers for which the condition (\ref{eq_C}) holds. 
For simplicity, denote $m:=\min\cN_{j+1}$. By the inductive assumption, we obtain $b_{m-1}\leq C a_{m-1}^2$, hence
\[b_m-b_{m-1}\geq C\lceil\max\{a_{m-1}^2,a_m\}\rceil+1-Ca_{m-1}^2\geq Ca_{m-1}^2+1-Ca_{m-1}^2\geq 1\]
so $b_{m-1}<b_m$, and, clearly, $b_m<b_{m+1}<\ldots<b_{\max\cN_{j+1}}$. 

Fix $1\leq l\leq|N_{j+1}|$. We have 
\[b_{m+l-1}\geq Ca_m=Ca_{m+l-1}\geq\frac{1}{C}a_{m+l-1}\]
so the first inequality in (\ref{eq_C}) holds for $k\in\cN_{j+1}$. Next, by assumption, we get
\begin{equation}\label{eq_a_mgeq}
a_{m+l-1}\geq 2(m+l-1), 
\end{equation}
whence 
\begin{equation}\label{eq_lleq}
l\leq a_{m-l+1}-m+1. 
\end{equation}
Consider two cases.
If $a_m\geq a_{m-1}^2$, then, from (\ref{eq_lleq})
\begin{align*}
b_{m-l+1}&=C\lceil a_m\rceil +l=C\lceil a_{m+l-1}\rceil +l\leq 2Ca_{m+l-1}+a_{m+l-1}-m+1\\
&\leq (2C+1)a_{m+l-1}\leq Ca_{m+l-1}^2,
\end{align*}
where the last inequality holds because $C\geq 1$ and, from (\ref{eq_a_mgeq}), we have
\[a_{m-l+1}\geq 2(m+l-1)\geq 2m\geq 2(k_0+1)\geq 4.\] 
Finally, if $a_{m-1}^2>a_m$, then, from (\ref{eq_a_mgeq}), we obtain (note that, by the definition of 
$\cN_{j}$ and $\cN_{j+1}$, we have $a_{m-1}<a_m$)
\begin{align*}
b_{m-l+1}&=C\lceil a_{m-1}^2\rceil+l\\
&\leq C\lceil (a_{m}-1)^2\rceil+l\\
&=C\lceil a_m^2-2a_m+1\rceil+l\\
&\leq C(a_m^2-2a_m+2)+l\\
&\leq Ca_m^2-2Ca_m+2C+Cl\\
&=Ca_{m+l-1}^2-C(2a_{m+l-1}-2-l)\\
&\leq Ca_{m+l-1}^2-C(4(m+l-1)-2-l)\\
&=Ca_{m+l-1}^2-C(4m+3l-6)\leq Ca_{m+l-1}^2.
\end{align*}
Hence we have shown that the second inequality in (\ref{eq_C}) holds for $k\in\cN_{j+1}$, and the proof is complete.
\bqed

\begin{Prop}\label{prop_characterization_commutative*-subalgebras}
Let $E$ be an infinite-dimensional closed commutative ${}^*$-subalgebra of $\ldss$ and let $(\sum_{j\in\cN_k}\langle\cdot,f_j\rangle f_j)_{k\in\N}$ 
be its canonical Schauder basis.
Moreover, let $(n_k)_{k\in\N}$ be a strictly increasing sequence of natural numbers and let $F$ be the closed ${}^*$-subalgebra of $s$ generated 
by $\{e_{n_k}\}_{k\in\N}$. 
Then the following assertions are equivalent:
\begin{enumerate}[\upshape (i)]
\item $E$ is isomorphic to $F$ as a Fr\'echet ${}^*$-algebra;
\item $\lambda^\infty(\max_{j\in\cN_k}|f_j|_q)\cong\lambda^\infty(n_k^q)$ as Fr\'echet ${}^*$-algebras;
\item there is a bijection $\sigma\colon\N\to\N$ such that $\lambda^\infty(\max_{j\in\cN_{\sigma(k)}}|f_j|_q)=\lambda^\infty(n_k^q)$ 
as Fr\'echet \linebreak ${}^*$-algebras;
\item there is a bijection $\sigma\colon\N\to\N$ such that $\lambda^\infty(\max_{j\in\cN_{\sigma(k)}}|f_j|_q)=\lambda^\infty(n_k^q)$ 
as sets;
\item there is a bijection $\sigma\colon\N\to\N$ such that
\begin{itemize}
 \item[($\alpha$)] $\forall q\in\N_0\;\exists r\in\N_0\;\exists C>0\;\forall k\in\N\quad\max_{j\in\cN_{\sigma(k)}}|f_j|_q\leq Cn_k^r$,
 \item[($\beta$)] $\forall r'\in\N_0\;\exists q'\in\N_0\;\exists C'>0\;\forall k\in\N\quad n_k^{r'}\leq C'\max_{j\in\cN_{\sigma(k)}}|f_j|_{q'}$.
\end{itemize}
\end{enumerate}
\end{Prop}
\bproof
This is an immediate consequence of Proposition \ref{prop_isomorphisms_kothe_algebras} and Corollary \ref{cor_comm_subalg_max|f_j|2}.
\bqed

\begin{Rem}
In view of Corollary \ref{cor_subalg_s2}, every closed ${}^*$-subalgebra of $s$ is isomorphic as a Fr\'echet ${}^*$-algebra to $\lambda^\infty(n_k^q)$
(i.e. the closed ${}^*$-subalgebra of $s$ generated by $\{e_{n_k}\}_{k\in\N}$)
for some strictly increasing sequence $(n_k)_{k\in\N}\subset\N$, hence Proposition \ref{prop_characterization_commutative*-subalgebras}
characterizes closed commutative ${}^*$-subalgebras of $\ldss$ which are isomorphic as Fr\'echet ${}^*$-algebras to some 
${}^*$-subalgebra of $s$.
\end{Rem}

The property (DN) for the space $s$ gives us the following inequality. 
\begin{Prop}\label{prop_l_2_DN_in_s2}
For every $p,r\in\N_0$ there is $q\in\N_0$ such that for all $\xi\in s$ with $||\xi||_{\ell_2}=1$ the following inequality holds 
\[|\xi|_p^r\leq|\xi|_q.\]
\end{Prop}
\begin{proof}
Take $p,r\in\N_0$ and let $j\in\N_0$ be such that $r\leq 2^j$. Applying iteratively ($j$-times) the inequality from Proposition 
\ref{prop_l_2_DN_in_s} to 
$\xi\in s$ with $||\xi||_{\ell_2}=1$ we get
\[|\xi|_p^r\leq|\xi|_p^{2^j}\leq|\xi|_{2^jp},\]
and thus the required inequality holds for $q=2^jp$.
\end{proof}

Now we are ready to prove Theorem \ref{th_complemented_subalg}.

\vspace{7pt}
\noindent
{\textbf{Proof of Theorem \ref{th_complemented_subalg}.}
(i)$\Rightarrow$(ii): By Corollary \ref{cor_subalg_s2}, each closed ${}^*$-subalgebra of $s$ is a complemented subspace of $s$.

(ii)$\Leftrightarrow$(iii): See e.g. \cite[Prop. 31.7]{MeV}.  

(iii)$\Rightarrow$(iv):
By Corollary \ref{cor_comm_subalg_max|f_j|2} and nuclearity (see e.g. \cite[Prop. 28.16]{MeV}), 
\[E\cong\lambda^\infty\bigg(\max_{j\in\cN_k}|f_j|_q\bigg)=\lambda^1\bigg(\max_{j\in\cN_k}|f_j|_q\bigg)\]
as Fr\'echet ${}^*$-algebras. 
Hence, by \cite[Prop. 5.3]{V4}, the property $(\Omega)$ yields
\[\forall l\;\exists m\;\forall n\;\exists t\;\exists C>0\;
\forall k\quad \max_{j\in\cN_k}|f_j|_l^t\max_{j\in\cN_k}|f_j|_n\leq C\max_{j\in\cN_k}|f_j|_m^{t+1}.\]
In particular, taking $l=0$, we get (iv).

(iv)$\Rightarrow$(i):
Take $p$ from the condition (iv). By Lemma \ref{lem_f_k_nuclearity}(ii), there is $p_1\geq p$ and a bijection $\sigma\colon\N\to\N$ such that
$(\max_{j\in\cN_{\sigma(k)}}|f_j|_{p_1})_{k\in\N}$ is non-decreasing and $\lim_{k\to\infty}\frac{k}{\max_{j\in\cN_{\sigma(k)}}|f_j|_{p_1}}=0$.  
Consequently, for $k$ big enough 
\[\max_{j\in\cN_{\sigma(k)}}|f_j|_{p_1}\geq 2k,\]
and therefore, by Lemma \ref{lem_a_k}, there is a strictly increasing sequence $(n_k)_{k\in\N}\subset\N$ and $C_1>0$ such that
\begin{equation}\label{eq_n_k2}
\frac{1}{C_1}\max_{j\in\cN_{\sigma(k)}}|f_j|_{p_1}\leq n_k\leq C_1\max_{j\in\cN_{\sigma(k)}}|f_j|_{p_1}^2
\end{equation}
for every $k\in\N$. Now, by the conditions (iv) and (\ref{eq_n_k2}), we get that for all $q$ there is $r$ and $C_2:=CC_1^r$ such that 
\[\max_{j\in\cN_{\sigma(k)}}|f_j|_{q}\leq C\max_{j\in\cN_{\sigma(k)}}|f_j|_{p_1}^r\leq C_2n_k^r\]
for all $k\in\N$, so the condition ($\alpha$) from Proposition \ref{prop_characterization_commutative*-subalgebras}(v) holds. 
Finally, by (\ref{eq_n_k2}) and Proposition \ref{prop_l_2_DN_in_s2} we obtain that for all $r'$ there is $q'$ and $C_3:=C_1^{r'}$ such that
\[n_k^{r'}\leq C_3\max_{j\in\cN_{\sigma(k)}}|f_j|_{p_1}^{2{r'}}\leq C_3\max_{j\in\cN_{\sigma(k)}}|f_j|_{q'} \]
for every $k\in\N$. Hence the condition ($\beta$) from Proposition \ref{prop_characterization_commutative*-subalgebras}(v) is satisfied, and therefore,
by Proposition \ref{prop_characterization_commutative*-subalgebras},
$E$ is isomorphic as a Fr\'echet ${}^*$-algebra to the closed ${}^*$-subalgebra of $s$ generated by $\{e_{n_k}\}_{k\in\N}$. 
\bqed

Now, we shall give an example of some class of closed commutative ${}^*$-subalgebras of $\ldss$ which are isomorphic to closed ${}^*$-subalgebras of $s$.

\begin{Exam}
\emph{
Let $\mathbb{H}_1:=[1]$. We define recursively \emph{Hadamard matrices}
\[\mathbb{H}_{2^n}:=
\left[\begin{array}{ll}
\mathbb{H}_{2^{n-1}} & \phantom{-}\mathbb{H}_{2^{n-1}}\\
\mathbb{H}_{2^{n-1}} & -\mathbb{H}_{2^{n-1}} 
\end{array} \right]\] 
for $n\in\N$. Then the matrices $\widehat{\mathbb{H}}_{2^n}:=2^{-\frac{n}{2}}\mathbb{H}_{2^n}$ are unitary, 
and thus their rows form an orthonormal system of $2^n$ vectors. Now fix an arbitrary sequence $(d_n)_{n\in\N}\subset\N_0$ and define
\[U:=
\left[\begin{array}{llll}
\widehat{\mathbb{H}}_{2^{d_1}} & 0                              & 0                              & \ldots\\
0                              & \widehat{\mathbb{H}}_{2^{d_2}} & 0                              & \ldots\\
0                              & 0                              & \widehat{\mathbb{H}}_{2^{d_3}} &       \\
\vdots                         & \vdots                         &                                & \ddots
\end{array} \right].\]
Let $f_k$ denote the $k$-th row of the matrix $U$. Then $(f_k)_{k\in\N}$ is an orthonormal basis of $\ell_2$ and clearly each $f_k$ belongs to $s$. We will show that the closed (maximal) commutative
${}^*$-subalgebra $\alg(\{\langle,\cdot,f_k\rangle f_k\}_{k\in\N})$ of $\ldss$ is isomorphic to some closed ${}^*$-subalgebra of $s$. By Theorem \ref{th_complemented_subalg}, it is enough to prove that
\begin{equation}\label{eq_hadamard}
\exists p\;\forall q\;\exists r\;\exists C>0\;\forall k\quad |f_k|_{\infty,q}\leq C|f_k|_{\infty,p}^r. 
\end{equation}
Fix $q\in\N_0$, $k\in\N$ and find $n\in\N$ such that $2^{d_1}+\ldots+2^{d_{n-1}}<k\leq 2^{d_1}+\ldots+2^{d_n}$. Then
\[\frac{|f_k|_{\infty,q}}{|f_k|_{\infty,1}^{2q}}=\frac{2^{-\frac{d_n}{2}}(2^{d_1}+\ldots+2^{d_n})^q}{2^{-d_nq}(2^{d_1}+\ldots+2^{d_n})^{2q}}
=2^{d_n(q-1/2)}(2^{d_1}+\ldots+2^{d_n})^{-q}\leq1\]
and thus the condition (\ref{eq_hadamard}) holds with $p=C=1$ and $r=2q$. }
\end{Exam}

The next theorem solves in negative \cite[Open Problem 4.13]{Cias}. In contrast to the algebra $s$, whose all closed ${}^*$-subalgebras are complemented subspaces of $s$ (Corollary \ref{cor_subalg_s2}),
Theorems \ref{th_complemented_subalg} and \ref{th_counterexample} imply that there is a closed commutative ${}^*$-subalgebra of $\ldss$ which is not complemented in $\ldss$ 
(otherewise it would have the property ($\Omega$), see \cite[Prop. 31.7]{MeV}). 
In the proof we will use the following identity.  

\begin{Lem}\label{lem_alphas}
For every increasing sequence $(\alpha_j)_{j\in\N}\subset(0,\infty)$ and every $p\in\N$ we have
\[\sup_{j\in\N}\left(\alpha_j^{p-j+1}\cdot\prod_{i=1}^{j-1}\alpha_i\right)=\prod_{i=1}^p \alpha_i.\]
\end{Lem}

\bproof
For $j\geq p+1$ we get
\[\frac{\alpha_j^{p-j+1}\cdot\prod_{i=1}^{j-1}\alpha_i}{\prod_{i=1}^p \alpha_i}=\alpha_j^{p-j+1}\cdot\prod_{i=p+1}^{j-1}\alpha_i
=\frac{\prod_{i=p+1}^{j-1}\alpha_i}{\alpha_j^{j-p-1}}\leq 1\]
and, similarly, for $j\leq p-1$ we obtain
\[\frac{\alpha_j^{p-j+1}\cdot\prod_{i=1}^{j-1}\alpha_i}{\prod_{i=1}^p \alpha_i}=\frac{\alpha_j^{p-j+1}}{\prod_{i=j}^p\alpha_i}\leq 1.\]
Since $\alpha_p^{p-p+1}\cdot\prod_{i=1}^{p-1}\alpha_i=\prod_{i=1}^p \alpha_i$, the supremum is attained for $j=p$, and we are done. 
\bqed

\begin{Th}\label{th_counterexample}
There is a closed commutative ${}^*$-subalgebra of $\ldss$ which is not isomorphic to any closed ${}^*$-subalgebra of $s$. 
\end{Th}

\bproof
Let $m_k$ be the $k$-th prime number, $N_{k,1}:=m_k$, $N_{k,j+1}:=m_k^{N_{k,j}}$ for $j,k\in\N$. Denote $a_{k,1}:=c_k$ and
\[a_{k,j}:=c_k\frac{\prod_{i=1}^{j-1}N_{k,i}}{N_{k,j}^{j-1}}\]
for $j\geq 2$, where the sequence $(c_k)_{k\in\N}$ is choosen so that $||(a_{k,j})_{j\in\N}||_{\ell_2}=1$, 
i.e. 
\[c_k:=\bigg(\sum_{j=1}^\infty\bigg(\frac{\prod_{i=1}^{j-1}N_{k,i}}{N_{k,j}^{j-1}}\bigg)^2\bigg)^{-1/2}.\]
The numbers $c_k$ are well-defined, because, by Lemma \ref{lem_alphas},
\begin{align*}
\sum_{j=1}^\infty\bigg(\frac{\prod_{i=1}^{j-1}N_{k,i}}{N_{k,j}^{j-1}}\bigg)^2
&=\sum_{j=1}^\infty\bigg(N_{k,j}^{-j+1}\cdot\prod_{i=1}^{j-1}N_{k,i}\bigg)^2
=\sum_{j=1}^\infty\frac{1}{N_{k,j}^2}\bigg(N_{k,j}^{1-j+1}\cdot\prod_{i=1}^{j-1}N_{k,i}\bigg)^2\\
&\leq\sup_{j\in\N}\bigg(N_{k,j}^{1-j+1}\cdot\prod_{i=1}^{j-1}N_{k,i}\bigg)^2\sum_{j=1}^\infty\frac{1}{N_{k,j}^2}
=N_{k,1}^2\sum_{j=1}^\infty\frac{1}{N_{k,j}^2}<N_{k,1}^2\sum_{j=1}^\infty\frac{1}{j^2}<\infty.
\end{align*}
Finally, define an orthonormal sequence $(f_k)_{k\in\N}$ by 
\[f_k:=\sum_{j=1}^\infty a_{k,j}e_{N_{k,j}}.\]
We will show that $\alg(\{\langle\cdot,f_k\rangle f_k\}_{k\in\N})$ is a closed ${}^*$-subalgebra of $\ldss$ which is not isomorphic as an algebra
to any closed ${}^*$-subalgebra of $s$.
By Theorem \ref{th_complemented_subalg} and nuclearity, it is enough to show that each $f_k$ belongs to $s$ and for every $p,r\in\N$ the 
following condition holds
\[\lim_{k\to\infty}\frac{|f_k|_{\infty,p+1}}{|f_k|_{\infty,p}^r}=\infty,\]
where $|\xi|_{\infty,q}:=\sup_{j\in\N}|\xi_j|j^q$.

Note first that $|f_k|_{\infty,p}=a_{k,p}N_{k,p}^p$. In fact, by Lemma \ref{lem_alphas}, we get
\begin{align*}
|f_k|_{\infty,p}&=\sup_{j\in\N}a_{k,j}N_{k,j}^p=c_k\sup_{j\in\N}\left(N_{k,j}^p\cdot\frac{\prod_{i=1}^{j-1}N_{k,i}}{N_{k,j}^{j-1}}\right)=
c_k\sup_{j\in\N}\left(N_{k,j}^{p-j+1}\cdot\prod_{i=1}^{j-1}N_{k,i}\right)\\
&=c_k\prod_{i=1}^p N_{k,i}=c_kN_{k,p}^p\cdot\frac{\prod_{i=1}^{p-1}N_{k,i}}{N_{k,p}^{p-1}}=a_{k,p}N_{k,p}^p.
\end{align*}
In particular, $f_k\in s$ for $k\in\N$.
Next, for $j,k\in\N$, we have
\[\frac{a_{k,j+1}N_{k,j+1}^j}{a_{k,j}}
=\frac{c_k N_{k,j+1}^j\cdot\frac{\prod_{i=1}^{j}N_{k,i}}{N_{k,j+1}^{j}}}{c_k\frac{\prod_{i=1}^{j-1}N_{k,i}}{N_{k,j}^{j-1}}}
=\frac{\prod_{i=1}^{j}N_{k,i}}{\frac{\prod_{i=1}^{j-1}N_{k,i}}{N_{k,j}^{j-1}}}=N_{k,j}^j.\]
Moreover, for every $j,r\in\N$ we get
\[\frac{N_{k,j+1}}{N_{k,j}^r}=\frac{m_k^{N_{k,j}}}{N_{k,j}^r}\geq \frac{2^{N_{k,j}}}{N_{k,j}^r}\xrightarrow[k\to\infty]{}\infty,\]
and clearly $a_{k,j}\leq 1$ for $j,k\in\N$. Hence, for $p,r\in\N$ we obtain
\begin{align*}
\frac{|f_k|_{\infty,p+1}}{|f_k|_{\infty,p}^r}&=\frac{a_{k,p+1}N_{k,p+1}^{p+1}}{a_{k,p}^rN_{k,p}^{pr}}
=\frac{a_{k,p+1}N_{k,p+1}^{p}}{a_{k,p}}\cdot\frac{1}{a_{k,p}^{r-1}}\cdot\frac{N_{k,p+1}}{N_{k,p}^{pr}}
=N_{k,p}^p\cdot\frac{1}{a_{k,p}^{r-1}}\cdot\frac{N_{k,p+1}}{N_{k,p}^{pr}}\\
&\geq\frac{N_{k,p+1}}{N_{k,p}^{pr}}\xrightarrow[k\to\infty]{}\infty,
\end{align*}
which is the desired conclusion.
\bqed

\subsection*{Acknowledgements}
I would like to thank Pawe{\l} Doma\'nski for his constant and generous support.

\vspace{1cm}
\begin{minipage}{7.5cm}
T. Cia\'s

Faculty of Mathematics and Comp. Sci.

A. Mickiewicz University in Pozna{\'n}

Umultowska 87

61-614 Pozna{\'n}, POLAND

e-mail: tcias@amu.edu.pl
\end{minipage}\

\end{document}